\documentclass[12pt]{amsart}
\usepackage{amsmath,amsfonts,amssymb}
\numberwithin{equation}{section}

\newtheorem{theorem}{Theorem}[section]
\newtheorem{lemma}[theorem]{Lemma}
\newtheorem{corollary}[theorem]{Corollary}
\newtheorem{proposition}[theorem]{Proposition}

\DeclareMathOperator{\Det}{Det}

 \DeclareMathOperator{\Rea}{Re}

\DeclareMathOperator{\Arg}{Arg}

\title [A Differential Characterization...]{ A Differential Characterization of the Metaplectic Kernel}
\author {Benjamin Cahen}
\address{Universit\'e de Lorraine, Site de Metz, UFR-MIM,
D\'epartement de math\'ematiques,
B\^atiment A,
3 rue Augustin Fresnel, BP 45112,
57073 METZ Cedex 03, France.}
\email{benjamin.cahen@univ-lorraine.fr}


\begin{document}

\maketitle

\begin{abstract} We present a derivation of the integral kernel
of a metaplectic representation operator in the Bargmann-Fock model based solely on
its defining intertwining property with the Heisenberg representation. 
Expressing a metaplectic operator as an integral operator transforms the corresponding infinitesimal intertwining identities
into a system of first-order partial differential equations satisfied by its kernel. 
We show that this system determines the kernel (up to a unit scalar) which is precisely the classical 
Gaussian associated with the symplectic transformation. A similar result is obtained for the  Schr\"odinger model.
\end{abstract}

\vspace{1cm}

\section {Introduction} \label{sec:intro}
The metaplectic representation (also called oscillator representation or Weil representation) is
a projective unitary representation of the symplectic group $Sp(n,{\mathbb R})$ which was first investigated by I. E. Segal,
D. Shale and A. Weil, see  \cite{Fo} and its references. The metaplectic representation plays a central role in different aeras of mathematics such as harmonic analysis (representation theory, time-frequency analysis), number theory (automorphic forms, theta functions) and mathematical physics (optics, quantum mechanics).

We can realize $Sp(n,{\mathbb R})$ as a subgroup $S$ of $SU(n,n)$, see \cite{Fo}, p. 175.  The group $S$ acts naturally on the $(2n+1)$-dimensional
(real) Heisenberg group $H_n$ and then acts on the generic representations (that is, the non-degenerate unitary irreducible representations) of $H_n$ in the Fock space $\mathcal F$.
Let $k\cdot \rho$ denote the action of 
$k\in S$ on the generic representation $ \rho$ of $H_n$.
Then $k\cdot  \rho$
and $ \rho$ are unitarily equivalent representations and there exists a unitary operator $\sigma(k)$ on $\mathcal F$ (defined up to a unit complex number) such that 
\begin{equation}\label{eq:inter}(k\cdot  \rho)(h)\sigma(k)=\sigma(k) \rho(h)\end{equation}
for each $h\in H_n$. The map $\sigma$ is thus the metaplectic representation of $S$.

The operators $\sigma(k)$, $k\in S$,  are integral operators whose kernels
 are Gaussian functions, see \cite{Fo} for instance. These kernels are classical objects, and many derivations of their explicit forms can be found in the literature. Most of these approaches either construct the kernels by solving Schr\"odinger evolution equation with quadratic Hamiltonians, or by exploiting their analytic properties and the covariance properties of coherent states \cite{CR, Fo, Itz}. In \cite{CaComp, CaExt, CaMR}, following a suggestion of \cite{Ne}, we have developed a more original method for constructing the kernels based on holomorphic representations of the Jacobi group \cite{Bern, CaPad}.

The aim of this note  is to present a different point of view, based exclusively on the defining intertwining property of the metaplectic representation (see Equation \ref{eq:inter}).
By differentiating  Equation \ref{eq:inter} with respect to $h$, we obtain the corresponding infinitesimal intertwining relations. Expressing $\sigma(k)$ as an integral operator then transforms these operator identities into a system of first-order partial differential equations satisfied by its integral kernel.

Our main observation is that this system completely characterizes the kernel up to a unit complex number. Solving the differential equations yields directly the classical Gaussian kernel, without introducing any a priori ansatz on its form. Then the quadratic exponential appears as a consequence of the infinitesimal covariance properties of the Heisenberg representation rather than as an assumed functional form. 

This method is not fondamentally original but, surprisingly, it is difficult to find it described in full detail in the literature. Many authors simply state the kernel and show that Equation \ref{eq:inter} is then satisfied.
Here we prove that the quadratic exponential is forced by the intertwining property.
Beyond recovering the classical formula, the interest of this approach lies in its simplicity. The derivation relies only on the infinitesimal action of the Heisenberg Lie algebra together with the intertwining property defining the metaplectic representation. It avoids sophisticated arguments. Moreover, the procedure is not specific to the metaplectic representation: whenever an intertwining operator between two realizations of a given representation admits an integral kernel, the same principle produces a system of differential equations characterizing that kernel. From this perspective, the metaplectic representation serves as a natural example of this general method.

The note is organized as follows. In Section \ref{sec:2}, we recall the Bargmann–Fock model and the Schr\"odinger model of the generic representations of $H_n$ and the realization of the symplectic group acting by automorphisms on $H_n$. In Section \ref{sec:3}, we derive the infinitesimal intertwining relations in the Bargmann–Fock model and we translate them into a system of partial differential equations for the integral kernel. This system is then solved explicitly in Section \ref{sec:4} and we recover the classical expression for the metaplectic kernel. In Section \ref{sec:5}, we apply the same method in the Schr\"odinger model. In Section \ref{sec:6}, we use the Bergmann transform in order to recover the expression of the metaplectic kernel in the  Schr\"odinger model from its expression in the Bargmann–Fock model. Finally,
Section \ref{sec:7} is devoted to some remarks and to presenting some further developments. 

\section{The generic representations of the Heisenberg group} \label{sec:2}

In this section, we first review some general facts on the Heisenberg group and its generic representations \cite{Fo, Tay}. We closely follow the presentation
of \cite{CaComp}.

For each $z,\,w \in {\mathbb C}^{n}$, we denote $zw:=\sum_{k=1}^nz_kw_k$. For each $z, z',w,w'\in {\mathbb C}^{n}$, let
\begin{equation*}\omega
((z,w),(z',w'))=\tfrac{i}{2}(zw'-z'w).
\end{equation*}

Then the $(2n+1)$-dimensional real Heisenberg group is 
\begin{equation*}H_n:=\{((z,{\bar
z}),c)\,:\,z\in {\mathbb C}^n, c\in {\mathbb R}\}\end{equation*} 
equipped with the
multiplication law
\begin{equation*}((z,{\bar
z}),c)\cdot ((z',{\bar z'}),c')=((z+z',{\bar z}+{\bar
z'}),c+c'+\tfrac{1}{2}\omega ((z,{\bar z}),(z',{\bar
z'}))).\end{equation*} 

Let $\lambda>0$. By the Stone-von Neumann
theorem, there exists a unique (up to unitary equivalence) unitary
irreducible representation $\rho_{\lambda}$ of $H_n$ whose restriction to the center
of $H_n$ is the character $((0,0),c)\rightarrow e^{i\lambda c}$
\cite{Tay}. 
The Bargmann-Fock realization of $\rho_{\lambda}$ is defined as follows \cite{Barg}. 

Let ${\mathcal F}_{\lambda}$ be the Hilbert space of all holomorphic functions $f$
on ${\mathbb C}^n$ such that \begin{equation*}\Vert f\Vert^2_{{\mathcal F}_{\lambda}}
:=\int_{{\mathbb C}^n} \vert f(z)\vert^2\, e^{-\lambda \vert
z\vert^2/2}\,d\mu_{\lambda} (z) <+\infty\end{equation*} where
 $d\mu_{\lambda}(z):=(2\pi
)^{-n}{\lambda}^n\,dm(z)$. Here $z=x+iy$ with $x$ and $y$ in ${\mathbb
R}^n$ and $dm(z):= dx\,dy$ stands for the
 Lebesgue measure on ${\mathbb
C}^n$.

Then 
\begin{equation*}({\rho}_{\lambda}(h)f)(z)=\exp \left(i\lambda c_0+\tfrac{\lambda}{2}{\bar z_0}z-\tfrac{\lambda}{4}\vert z_0\vert^2\right)\,f(z- z_0) \end{equation*}
for each $h=((z_0,{\bar z_0}), c_0)\in H_n$ and $z\in {\mathbb C}^n$.

For each $z\in {\mathbb C}^n$, consider the \textit {coherent state}
$e_z(w)=\exp (\lambda{\bar z}w/2)$. Then we have the reproducing property
$f(z)=\langle f,e_z\rangle_{{\mathcal F}_{\lambda}}$ for each $f\in {\mathcal
F}_{\lambda}$.

Let $A$ be an operator $A$ on ${\mathcal
F}_{\lambda}$. Then we have

\begin{align*}
A\,f(z)&=\langle A\,f\,,\,e_z \rangle_{{\mathcal F}_{\lambda}}  
  =\langle f\,,\,A^{\ast}\,e_z\rangle_{{\mathcal F}_{\lambda}}  \\
&=\int _{{\mathbb C}^n}\,f(w)\overline {A^{\ast}\,e_z(w)}\,e^{-\lambda \vert
w\vert^2/2}\,d\mu_{\lambda} (w) \\
&=\int _{{\mathbb C}^n}\,f(w)\overline {\langle A^{\ast}\,e_z,e_w\rangle_{{\mathcal F}_{\lambda}}}\,e^{-\lambda \vert
w\vert^2/2}\,d\mu_{\lambda} (w) \\
&=\int _{{\mathbb C}^n} \,f(w)\,\langle Ae_w,e_z\rangle_{{\mathcal F}_{\lambda}}\,e^{-\lambda \vert
w\vert^2/2}\,d\mu_{\lambda}(w).\\ 
\end{align*}
Then we see that $A$ has integral kernel 
\begin{equation} k_A(z,w):=\langle Ae_w,e_z\rangle_{{\mathcal F}_{\lambda}}.\end{equation}

Remark that $k_A(z,w)$ is holomorphic in $z$ and anti-holomorphic in $w$.

Now,  we consider
another realization of the unitary irreducible representation of $H_n$ with central character
$((0,0),c)\rightarrow e^{i\lambda c}$, namely the Schr\"odinger representation $\rho'_{\lambda}$
defined on $L^2({\mathbb R}^n)$ by
\begin{equation*}(\rho'_{\lambda}((a+ib,a-ib),c)\phi)(x)
=\exp \left(i\lambda (c-bx+\tfrac{1}{2}ab)\right)\,\phi(x-a) \end{equation*}
for each $a, b, x \in {\mathbb R}^n$.

Note that in the setting of the orbit method, the Schr\"odinger realization
can be obtained from a real polarization of the corresponding
coadjoint orbit of $H_n$ and the Bargmann-Fock realization from a
complex polarization \cite{Kir}.

Let's consider the differentials of $\rho_{\lambda}$ and $\rho'_{\lambda}$. Let ${\mathfrak h}_n$ be the Lie algebra of $H_n$ and let  ${\mathfrak h}_n^c$ be the complexification of  ${\mathfrak h}_n$. We write the elements of ${\mathfrak h}_n^c$ as $[(z,w),c]$ where $z,w\in {\mathbb C}^n$ and $c\in {\mathbb C}$. Then the Lie brackets of ${\mathfrak h}_n^c$ are given by
\begin{equation*}\bigl[[(z,w),c],[(z',w'),c']\bigr]=[(0,0),\omega((z,w),(z',w'))]\end{equation*}
where $z,w, z', w'\in {\mathbb C}^n$ and $c, c'\in {\mathbb C}$. Note that the exponential map
$\exp: {\mathfrak h}_n \rightarrow H_n$ is given by $\exp[(z,{\bar z}), c]=((z,{\bar z}), c)$.

The differential $d\rho_{\lambda}$ of $\rho_{\lambda}$ can be extended to ${\mathfrak h}_n^c$
by linearity. We easily verify that, for each $[(u,v),c]\in  {\mathfrak h}_n^c$, $f\in {\mathcal F}_{\lambda}$ and $z\in {\mathbb C}$, we have
\begin{equation}\label{eq:drho} (d\rho_{\lambda}([(u,v),c])f)(z)=
\left(i\lambda c+\tfrac{\lambda}{2}vz\right)f(z)-\sum_{j=1}^nu_j\frac{\partial f}{\partial z_j}.\end{equation}

Let us introduce the map $X\in {\mathfrak h}_n^c\rightarrow X^{\ast}\in {\mathfrak h}_n^c$
defined as follows. If $X=Y+iZ$ where $Y,Z \in {\mathfrak h}_n$ then $X^{\ast}=-Y+iZ$.

\begin{lemma} \label{lemstar} \begin{enumerate} 
\item The map $X\rightarrow X^{\ast}$ is an anti-automorphism of ${\mathfrak h}_n^c$, that is, we have $[X^{\ast}, Y^{\ast}]=-[X, Y]$ for each $X,Y \in {\mathfrak h}_n^c$;

\item If $X=[(u,v),c]\in {\mathfrak h}_n^c$ then we have $X^{\ast}=[(-\bar{v},-\bar{u}), -\bar{c}]$;

\item For each $X\in  {\mathfrak h}_n^c$ and each $f,g \in{\mathcal F}_{\lambda}$, we have 
\begin{equation*}\langle d\rho_{\lambda}(X)f,g\rangle=\langle f, d\rho_{\lambda}(X^{\ast})g\rangle. \end{equation*}
\end{enumerate} \end{lemma}

\begin{proof} (1) is easy to verify. In order to prove (2), we have just to notice that 
each $X=[(u,v),c]\in {\mathfrak h}_n^c$ can be decomposed as $X=Y+iZ$ where $X=[(a, \bar{a}),c_1]\in  {\mathfrak h}_n$ and $Y=[(b, \bar{b}),c_2]\in  {\mathfrak h}_n$ are defined by
$a=\tfrac{1}{2}(u+\bar{v})$, $b=\tfrac{1}{2i}(u-\bar{v})$ and $c=c_1+ic_2$ with $c_1,c_2\in {\mathbb R}$. Finally, since $\rho_{\lambda}$ is unitary, (3) is clearly satisfied for each $X\in {\mathfrak h}_n$ and also for each $X\in i{\mathfrak h}_n$ hence for each 
$X\in  {\mathfrak h}_n^c$.
\end{proof}

The differential $d\rho_{\lambda}'$ of $\rho_{\lambda}'$ can be also easily computed. We get
\begin{equation}\label{eq:drhopr} (d\rho_{\lambda}'([(a+ib,a-ib),c])\phi)(x)=
i\lambda (c-bx)\phi(x)-\sum_{j=1}^n a_j\frac{\partial \phi}{\partial x_j}\end{equation}
for each $a,b\in {\mathbb R}^n$, $c\in {\mathbb R}$ and $\phi\in C^{\infty}( {\mathbb R}^n)$.

\section{The metaplectic representation} \label{sec:3}
The complex symplectic group $Sp(n, {\mathbb C})$ is the group of all automorphisms of
${\mathbb C}^{2n}$ that preserve $\omega$. Then each element of  $Sp(n, {\mathbb C})$ is a $(2n)\times (2n)$ complex matrix $M$ such that $M^tJM=J$ where
\begin{equation*}J:=\begin{pmatrix} 0&I_n\\
-I_n&0 \end{pmatrix}.\end{equation*}
Here the superscript 't' denotes transposition. We usually write $M$ as a block matrix
\begin{equation*}M=\begin{pmatrix} A&B\\
C&D \end{pmatrix}\end{equation*}
where $A,B,C$ and $D$ are $n\times n$ matrices. The condition  $M^tJM=J$ is then equivalent
to the following relations
\begin{equation}\label{eq:sympM} A^tD-C^tB=I_n; \quad A^tC=C^tA; \quad B^tD=D^tB. \end{equation}
Note that if $M\in Sp(n, {\mathbb C})$ then $M^t\in Sp(n, {\mathbb C})$. Indeed, the condition
$M^tJM=J$ implies that $(-M^tJ)(MJ)=I_n$. This shows that the matrices $-M^tJ$ and $MJ$ are inverses of each other hence $(MJ)(-M^tJ)=I_n$ which gives $MJM^t=J$.

Consequently, by replacing $M$ by $M^t$ in the relations $(\ref{eq:sympM})$, we also have the following relations
\begin{equation}\label{eq:sympMbis} AD^t-BC^t=I_n; \quad AB^t=BA^t; \quad CD^t=DC^t. \end{equation}

Let \begin{equation*}U:=\begin{pmatrix} I_n&iI_n\\
I_n&-iI_n \end{pmatrix}.\end{equation*}
We denote by $S$ the image of the (real) symplectic group $Sp(n, {\mathbb R})$ under the map
$M\rightarrow UMU^{-1}$. Then the elements of $S$ are the symplectic matrices of the form
\begin{equation*}\begin{pmatrix} P&Q\\ \bar{Q}&\bar{P}\end{pmatrix},\end{equation*}
where $P$ and $Q$ are complex matrices, see \cite{Fo}. It is useful to have the following formulas.
Let  \begin{equation*}M=\begin{pmatrix} A&B\\
C&D \end{pmatrix} \in Sp(n, {\mathbb R})\end{equation*}
and \begin{equation*}M'=UMU^{-1}=\begin{pmatrix} P&Q\\ \bar{Q}&\bar{P}\end{pmatrix} \in S.
\end{equation*}
Then, by an easy computation, we get
\begin{equation*}P=\frac{1}{2}(A-iB+iC+D);\quad Q=\frac{1}{2}(A+iB+iC-D).
\end{equation*}

The group $Sp(n, {\mathbb C})$ acts on ${\mathbb C}^{2n}$ by
\begin{equation*}\begin{pmatrix} A&B\\
C&D \end{pmatrix} \cdot (z,w)=(Az+Bw, Cz+Dw). \end{equation*}

Note that here the elements of ${\mathbb C}^{n}$ (and of ${\mathbb C}^{2n}$) are considered
as column vectors.

The restriction of this action to $S$ induces an action of $S$ on ${\mathbb C}^{n}$:
\begin{equation*}\begin{pmatrix} P&Q\\
\bar{Q}&\bar{P} \end{pmatrix} \cdot (z,\bar{z})=(Pz+Q\bar{z}, \bar{P}\bar{z}+\bar{Q}z) \end{equation*} which, in turn, induces an action of $S$ on $H_n$:
\begin{equation*}k\cdot ((z,{\bar z}),c)=(k\cdot (z,{\bar
z}),c).\end{equation*} By differentiating this last action, we also obtain an action of $S$ on 
${\mathfrak h}_n$.

Fix $\lambda >0$ and let $\rho:=\rho_{\lambda}$. For each $k\in S$, we define
$\rho_k$ by $\rho_k(h):=\rho(k\cdot h)$ for each $h\in H_n$. Since $k\in S$ acts on $H_n$
as a group isomorphism, we see that $\rho_k$ is also a generic representation of $H_n$
which have the same central character as $\rho$. Consequently, by the Stone-von Neumann
theorem, $\rho_k$ and $\rho$ are unitarily equivalent, that is, there exists a unitary operator
$A_k$ of ${\mathcal F}_{\lambda}$ (defined up to a unit complex number) such that
\begin{equation}\label{eq:inter}\rho_k(h)=A_k\rho (h)A_k^{-1}\end{equation}
for each $h\in H_n$. 

The map $k\rightarrow A_k$ is then a projective representation of $S$ called \textit{ the metaplectic representation}. By differentiating $(\ref{eq:inter})$, we get
\begin{equation}\label{eq:interd}d\rho(k\cdot X)A_k=A_kd\rho (X)\end{equation}
for each $X\in {\mathfrak h}_n$. 

We can reformulate this relation in terms of integral kernels. Let us denote by $b_k(z,w)=k_{A_k}(z,w)$ the integral kernel of $A_k$, that is, we have
\begin{equation}(A_kf)(z)=\int_{{\mathbb C}^n} b_k(z,w)\,f(w)\,e^{-\lambda \vert
w\vert^2/2}\,d\mu_{\lambda}(w)\end{equation}
for each $f\in {\mathcal F}_{\lambda}$ and each $z\in {\mathbb C}^{n}$.

\begin{proposition} \label{prop:ker} Let $k=\left(\begin{smallmatrix} A&B\\
C&D \end{smallmatrix}\right)\in S$. Then $A_k$ satisfies the relation (\ref{eq:interd}) if and only if
its integral kernel $b_k(z,w)$ satisfies the system
\begin{equation}\label{eq:ker} \begin{split}\frac{\lambda}{2}(Cu+Dv)&zb_k(z,w)-\sum_{j=1}^n(Au+Bv)_j
\frac{\partial}{\partial z_j}b_k(z,w)\\
&=-\frac{\lambda}{2}u\bar{w}b_k(z,w)+\sum_{j=1}^nv_j\frac{\partial}{\partial \bar{w}_j}b_k(z,w)\end{split}\end{equation}
for each $u,v,z,w \in {\mathbb C}^{n}$.
\end{proposition}

\begin{proof} Let $k=\left(\begin{smallmatrix} A&B\\
C&D \end{smallmatrix}\right)\in S$ and $X=[(u,v),0]\in  {\mathfrak h}_n^c$. On the one hand, we have
\begin{align*} (d\rho(&k\cdot X)A_k f)(z)\\
&=\frac{\lambda}{2}(Cu+Dv)z(A_kf)(z)-
\sum_{j=1}^n(Au+Bv)_j\frac{\partial (A_kf)}{\partial z_j}\\
&=\frac{\lambda}{2}(Cu+Dv)z\,\int _{{\mathbb C}^n} b_k(z,w)\,f(w)\,e^{-\lambda \vert
w\vert^2/2}\,d\mu_{\lambda}(w)\\
&\qquad-\sum_{j=1}^n(Au+Bv)_j\, \int _{{\mathbb C}^n} \frac{\partial}{\partial z_j}(b_k(z,w))
\,f(w)\,e^{-\lambda \vert
w\vert^2/2}\,d\mu_{\lambda}(w),
\end{align*}  
for each  $f\in {\mathcal F}_{\lambda}$ and $z\in {\mathbb C}^n$.

On the other hand, by the reproducing property and (3) of Lemma \ref{lemstar}, we have
\begin{equation*}(A_kd\rho(X)f)(z)=\langle A_kd\rho(X)f, e_z\rangle_{{\mathcal F}_{\lambda}}
=\langle f, d\rho(X^{\ast})A_k^{\ast}e_z\rangle_{{\mathcal F}_{\lambda}}.\end{equation*}

Writing
\begin{equation*}(A_k^{\ast}e_z)(w)=\langle A_k^{\ast}e_z, e_w\rangle_{{\mathcal F}_{\lambda}}
=\langle e_z, A_ke_w\rangle_{{\mathcal F}_{\lambda}}=\overline{b_k(z,w)},
\end{equation*}
we see that the integral kernel of $A_kd\rho(X)$ is the conjugate of 
\begin{equation*}(d\rho(X^{\ast})A_k^{\ast}e_z)(w)=(d\rho(X^{\ast})(\overline{b_k(z,\cdot)}))(w),
\end{equation*}
that is,
\begin{equation*}-\frac{\lambda}{2}u\bar{w}b_k(z,w)+\sum_{j=1}^n v_j\frac{\partial}{\partial \bar{w}_j}b_k(z,w).
\end{equation*}
The result hence follows.
\end{proof}

\section{Determination of the metaplectic kernel} \label{sec:4}
In this section, we recover the explicit form of $b_k(z,w)$ by solving the system (\ref{eq:ker}).

\begin{proposition} \label{propformbk} Let $k=\left(\begin{smallmatrix} A&B\\
C&D \end{smallmatrix}\right)\in S$. Then we have
\begin{equation*} b_k(z,w)=c_k\,\exp\left( \frac{\lambda}{4}zCA^{-1}z+\frac{\lambda}{2}(A^{-1}z)\bar{w}-\frac{\lambda}{4}\bar{w}(A^{-1}B\bar{w})\right)\end{equation*}
where $c_k\in {\mathbb C}$.
\end{proposition}

\begin{proof} Let $k=\left(\begin{smallmatrix} A&B\\
C&D \end{smallmatrix}\right)\in S$. In order to prove the proposition, one can verify that the
above expression for $b_k(z,w)$ gives a solution of (\ref{eq:ker}) but this method is not in the spirit
of this note. Then, we start from (\ref{eq:ker}) and write
\begin{equation*}b_k(z,w)=\exp(\beta_k(z,w)).\end{equation*}
thus (\ref{eq:ker}) becomes the following relation
\begin{equation}\label{eq:kerbis} \begin{split}\frac{\lambda}{2}(Cu+Dv)&z-\sum_{j=1}^n(Au+Bv)_j
\frac{\partial}{\partial z_j}\beta_k(z,w)\\
&=-\frac{\lambda}{2}u\bar{w}+\sum_{j=1}^nv_j\frac{\partial}{\partial \bar{w}_j}\beta_k(z,w)\end{split}\end{equation}
for each $u,v,z,w \in {\mathbb C}^{n}$.

Now, let $(e_1,e_2,\ldots,e_n)$ be the standard basis of ${\mathbb C}^{n}$. For each function $F(z,w)$
on ${\mathbb C}^{2n}$ which is holomorphic in $z$ and anti-holomorphic in $\bar{w}$, we define
\begin{equation*}(\nabla_1\, F)(z,w):=\sum_{j=1}^n\frac{\partial}{\partial z_j}F(z,w)\,e_j\end{equation*}
and
\begin{equation*}(\nabla_2\, F)(z,w):=\sum_{j=1}^n \frac{\partial}{\partial \bar{w}_j}F(z,w)\,e_j.\end{equation*}

By taking $v=0$ in (\ref{eq:kerbis}), we get
\begin{equation*}\frac{\lambda}{2}u(C^tz)-uA^t(\nabla_1 \beta_k)(z,w)=-\frac{\lambda}{2}u\bar{w}
\end{equation*} for each $u,z,w\in {\mathbb C}^{n}$, hence
\begin{equation*}\frac{\lambda}{2}C^tz-A^t(\nabla_1 \beta_k)(z,w)=-\frac{\lambda}{2}\bar{w}
\end{equation*} for each $z,w\in {\mathbb C}^{n}$. This gives
\begin{equation}\label{eq:141} (\nabla_1\, \beta_k)(z,w)=\frac{\lambda}{2}((A^t)^{-1}C^tz+(A^t)^{-1}\bar{w})=\frac{\lambda}{2}(CA^{-1}z+(A^t)^{-1}\bar{w})\end{equation}
since $CA^{-1}$ is symmetric, see (\ref{eq:sympM}). Consequently, we can write for each $j=1,2,\ldots,n$
\begin{align*}
\frac{\partial}{\partial z_j}\beta_k(z,w)=&\frac{\lambda}{2}e_j\,(CA^{-1}z+(A^t)^{-1}\bar{w})\\
=&\frac{\partial}{\partial z_j}\left(\frac{\lambda}{4}zCA^{-1}z+\frac{\lambda}{2}(A^{-1}z)
\bar{w}\right).\\
\end{align*}
This implies that there exists a function $\gamma_k(\bar{w})$ such that
\begin{equation}\label{eq:241} \beta_k(z,w)=\frac{\lambda}{4}zCA^{-1}z+
\frac{\lambda}{2}(A^{-1}z)\bar{w}+\gamma_k(\bar{w}).
\end{equation}
On the other hand, by taking $u=0$ in (\ref{eq:kerbis}) we obtain
\begin{equation*}\frac{\lambda}{2}(Dv)z-\sum_{j=1}^n(Bv)_j\frac{\partial}{\partial z_j}\beta_k(z,w)=\sum_{j=1}^n v_j \frac{\partial}{\partial {\bar w_j}}\beta_k(z,w)
\end{equation*} for each $v, z,w\in {\mathbb C}^{n}$ and by replacing in this equation
the expression of $\beta_k(z,w)$ given by (\ref{eq:241}) we find that
\begin{equation*}\frac{\lambda}{2}(Dv)z-\frac{\lambda}{2}(Bv)(CA^{-1}z+(A^t)^{-1}\bar{w})=
\frac{\lambda}{2}(A^{-1}z)v+v\nabla_2\gamma_k({\bar w})
\end{equation*} for each $v, z,w\in {\mathbb C}^{n}$. Then, taking the relation
$D^tA-B^tC=I_n$ into account (see (\ref{eq:sympM})), the preceding equation
can be simplified as
\begin{equation*}\nabla_2\gamma_k({\bar w})=-\frac{\lambda}{2}B^t(A^t)^{-1}
\bar{w}. \end{equation*} 
Thus, since $A^{-1}B$ is symmetric, we get
\begin{equation*}\frac{\partial \gamma_k}{\partial {\bar w}_j}=-\frac{\lambda}{2}(A^{-1}B{\bar w})_j=-\frac{\lambda}{4}\frac{\partial }{\partial {\bar w}_j}({\bar w}(A^{-1}B{\bar w}))
\end{equation*} for each $j=1,2,\ldots,n$. Hence there exists a constant $\epsilon_k$ such that
\begin{equation*}\gamma_k({\bar w})=-\frac{\lambda}{4}{\bar w}(A^{-1}B{\bar w})+\epsilon_k
\end{equation*} and by replacing in (\ref{eq:241}) we obtain the desired result.
\end{proof}
In order to complete the expression for the metaplectic kernel given in Proposition
\ref{propformbk}, we aim to determine $c_k$ up to a unit complex number by expressing that $A_k$ is unitary. To this goal, we need to compute some Gaussian integrals. The following lemma is just a variant of \cite[ Theorem 3, p. 258]{Fo}, see \cite{CaComp}.
Note that for $z\in \mathbb C$ we define $z^{1/2}$ as the principal determination of the square-root (with branch cut along the negative real axis).

\begin{lemma} \label{lemgauss} Let $A,B, D$ be $n\times n$ complex matrices such that
$A^t=A, D^t=D$. Let $M=\bigl(\begin{smallmatrix} A&B^t\\
B&D \end{smallmatrix}\bigr)$, $U=\bigl(\begin{smallmatrix} I_n&iI_n\\
I_n&-iI_n \end{smallmatrix}\bigr)$ and $N=U^tMU$. Assume that $\Rea (N)$ is positive definite. Let $u,v \in {\mathbb C}^n$. Then we have 
\begin{align*}\int_{{\mathbb C}^n}&\exp\left(-\left( w(Aw)+{\bar w}(D{\bar w})+2{\bar w}(Bw)\right)\right)\exp (uw+v{\bar w})\,dm(w)\\
=&\pi^n(\Det N)^{-1/2}\exp \left( \tfrac{1}{4}\begin{pmatrix}u&v\end{pmatrix}M^{-1}\begin{pmatrix}u\\v\end{pmatrix}\right).
\end{align*}
\end{lemma}

\begin{proof} Write $w=x+iy$ with $x, y \in {\mathbb R}^n$. Then $\begin{pmatrix}w\\{\bar w}\end{pmatrix}=U\begin{pmatrix}x\\y\end{pmatrix}$. We have
\begin{equation*}w(Aw)+{\bar w}(D{\bar w})+2{\bar w}(Bw)=\begin{pmatrix}w &{\bar w}\end{pmatrix}M\begin{pmatrix}w\\{\bar w}\end{pmatrix}=\begin{pmatrix}x&y \end{pmatrix}N\begin{pmatrix}x\\y\end{pmatrix}
\end{equation*} and $uw+v{\bar w}= \begin{pmatrix}u&v\end{pmatrix}U
\begin{pmatrix}x\\y\end{pmatrix}$.

The result is then a consequence of  the well-known equality
\begin{equation}\label{eq:Gauss} \int_{{\mathbb R}^n}\exp (-xAx+zx)\,dx=(\Det A)^{-1/2}\pi ^{n/2}
\exp\left(\tfrac{1}{4}z(A^{-1}z)\right)
\end{equation} where $z\in {\mathbb C}^n$ and $A$ is a  $n\times n$ symmetric complex matrix such that $\Rea(A)$ is definite positive.
\end{proof}

\begin{proposition} \label{propck} \begin{enumerate}
\item For each $k=\bigl(\begin{smallmatrix} P&Q\\
\bar Q&\bar P \end{smallmatrix}\bigr)\in S$, we have $c_k=(\Det\,P)^{-1/2}$ up to a unit scalar;
\item For each $k=\bigl(\begin{smallmatrix} P&Q\\
\bar Q&\bar P \end{smallmatrix}\bigr)\in S$ let us denote by $\sigma(k)$ the operator
on ${\mathcal F}_{\lambda}$ with integral kernel
\begin{equation*} b_k(z,w)=(\Det\,P)^{-1/2}\,\exp\left( \frac{\lambda}{4}z{\bar Q}P^{-1}z+\frac{\lambda}{2}(P^{-1}z)\bar{w}-\frac{\lambda}{4}\bar{w}(P^{-1}Q\bar{w})\right).\end{equation*} Then, for each $k\in S$, the operator $\sigma(k)$ is unitary and
for each $k,k'\in S$, we have $\sigma(kk')=\pm \sigma(k)\sigma(k')$.
\end{enumerate}
\end{proposition}

\begin{proof} 1) Let $k=\bigl(\begin{smallmatrix} P&Q\\
\bar Q&\bar P \end{smallmatrix}\bigr)\in S$. The integral kernel of $A_k$ being $b_k(z,w)$,
the integral kernel of $A_k^{\ast}$ is $\overline{b_k(w,z)}$. Moreover, the integral kernel of the identity operator of  ${\mathcal F}_{\lambda}$ is 
\begin{equation*}k_{\rm{Id}}(z,w)=\langle e_w,e_z\rangle_{{\mathcal F}_{\lambda}}=e_w(z).
\end{equation*}
Then the relation $A_kA_k^{\ast}=\rm{Id}$ can be translated in terms of kernels as
\begin{equation*}e_w(z)=\int_{{\mathbb C}^n}\,b_k(z,u)\overline{b_k(w,u)}
\,e^{ -\lambda\vert w\vert^2/2}\,d\mu_{\lambda}(w).\end{equation*}

By taking $z=w=0$ in this equality we get
\begin{equation*}1=\vert c_k\vert^2 \,\int_{{\mathbb C}^n}\,
\exp\left(-\tfrac{\lambda}{4}(u\overline{P^{-1}Q}u+{\bar u}P^{-1}Q{\bar u}+2u{\bar u})\right)\,d\mu_{\lambda}(u).
\end{equation*}
In order to use Lemma \ref{lemgauss}, we have to compute the determinant of the matrix
$N:=U^tMU$ where  
\begin{equation*}M:=\frac{\lambda}{4}\begin{pmatrix} \overline{P^{-1}Q}&I_n\\I_n&P^{-1}Q\end{pmatrix}. \end{equation*}
Then, since 
\begin{equation*}{\bar P}^{-1}{\bar Q}P^{-1}Q={\bar P}^{-1}{\bar Q}Q^t(P^t)^{-1}=
{\bar P}^{-1}({\bar P}P^t-I_n)(P^t)^{-1}=I_n-{\bar P}^{-1}(P^t)^{-1},\end{equation*}
we have
\begin{equation*} \Det(M)=\left(\frac{\lambda}{4}\right)^{2n}\Det(\overline{P^{-1}Q}P^{-1}Q-I_n)=
\left(\frac{\lambda}{4}\right)^{2n}(-1)^n\vert \Det(P)\vert^{-2},
\end{equation*} Hence
\begin{equation*} \Det(N)={\lambda}^{2n}2^{-2n}\vert \Det(P)\vert^{-2}\end{equation*}
and finally $\vert c_k\vert^2=\vert \Det(P)\vert^{-1}$. This proves that $c_k=\Det(P)^{-1/2}$
up to a unit scalar.

2) Let $k=\bigl(\begin{smallmatrix} P&Q\\
\bar Q&\bar P \end{smallmatrix}\bigr)\in S$, $k'=\bigl(\begin{smallmatrix} P'&Q'\\
\bar Q'&\bar P' \end{smallmatrix}\bigr)\in S$ and let $k''=kk'=\bigl(\begin{smallmatrix} P''&Q''\\
\bar Q''&\bar P'' \end{smallmatrix}\bigr)\in S$. First note that, by Schur's Lemma, there exists
a unit complex number $\alpha(k,k')$ such that $\sigma(k'')=\alpha(k,k')\sigma(k)\sigma(k')$.
Then by writing that the integral kernel of $\sigma(k)\sigma(k')$ is the convolution of the kernels
of $\sigma(k)$ and of $\sigma(k')$, we get
\begin{equation*}b_{k''}(z,w)=\alpha (k,k')\,\int_{{\mathbb C}^n}b_k(z,u)b_{k'}(u,w)
e^{-\lambda \vert
u\vert^2/2}\,d\mu_{\lambda} (u).
\end{equation*} Taking $z=w=0$ in this equality, we obtain
\begin{align*}(\Det P'')^{-1/2}=\alpha (k,k')&(\Det P)^{-1/2}(\Det P')^{-1/2}\\
\times &\int_{{\mathbb C}^n}\exp \left( \tfrac{\lambda}{4}\left(-{\bar u}P^{-1}Q{\bar u}+u {\bar Q'}
P'^{-1}u\right)\right)e^{-\lambda \vert
u\vert^2/2}\,d\mu_{\lambda} (u).
\end{align*}
Recall that $P^{-1}Q$ and ${\bar Q'}
P'^{-1}$ are symmetric. Then, the integral in the preceding equality can be computed by using Lemma \ref{lemgauss}. We find that its value is
\begin{equation*}\Det ^{-1/2}(I_n+P^{-1}Q{\bar Q'}
P'^{-1})=\Det ^{-1/2}(P^{-1}(PP'+Q{\bar Q'})P'^{-1})=
\Det ^{-1/2}(P^{-1}P''P'^{-1}).\end{equation*}
This implies that
\begin{equation*}(\Det P'')^{-1}=\alpha(k,k')^2(\Det P)^{-1}(\Det P')^{-1}(\Det (P^{-1}P''P'^{-1})^{-1}
\end{equation*} hence $\alpha(k,k')^2=1$ and $\alpha(k,k')=\pm 1$.
\end{proof}

\section{The metaplectic kernel in the Schr\"odinger model} \label{sec:5}
In this section, we use essentially the same arguments as in Section \ref{sec:3} and Section \ref{sec:4} in order to recover the expression of the metaplectic kernel in the case of the Schr\"odinger realization. However, for the calculations, the Schr\"odinger model
is less convenient than the Fock model. In particular, the integral kernels of operators
on $L^2({\mathbb R}^n)$ are in general only defined in the sense of distributions while the integral kernels of operators on the Fock space are holomorphic/anti-holomorphic functions.

Recall that we have introduced in Section \ref{sec:2} the Schr\"odinger  representation 
$\rho'_{\lambda}$ of $H_n$ on $L^2({\mathbb R}^n)$ defined by
\begin{equation*}(\rho'_{\lambda}((a+ib,a-ib),c)\phi)(x)
=\exp \left(i\lambda (c-bx+\tfrac{1}{2}ab)\right)\,\phi(x-a) \end{equation*}
for each $a, b, x \in {\mathbb R}^n$ and its differential $d\rho'_{\lambda}$ given by
\begin{equation*}(d\rho_{\lambda}'([(a+ib,a-ib),c])\phi)(x)=
i\lambda (c-bx)\phi(x)-\sum_{j=1}^n a_j\frac{\partial \phi}{\partial x_j}\end{equation*}
for each $a,b\in {\mathbb R}^n$, $c\in {\mathbb R}$ and $\phi$ element of the Schwartz space ${\mathcal S}( {\mathbb R}^n)$.

Consider the (natural) action of $Sp(n,{\mathbb R})$ on ${\mathbb R}^{2n}$ defined by
\begin{equation*}\begin{pmatrix} A&B\\
C&D \end{pmatrix} \cdot (a,b)=(Aa+Bb, Ca+Db). \end{equation*}
Note that this action corresponds to the action of $S$ on ${\mathbb C}^{n}$ which was introduced in Section \ref{sec:3}, via the map $g\rightarrow UgU^{-1}$ from $Sp(n,{\mathbb R})$
onto $S$. Indeed, if $k=UgU^{-1}\in S$ with $g=\left(\begin{matrix} A&B\\
C&D \end{matrix}\right) \in Sp(n,{\mathbb R})$ then we have
\begin{equation*}k\cdot (a+ib,a-ib)=(Aa+Bb+i(Ca+Db),Aa+Bb-i(Ca+Db))\end{equation*} for each $a,b\in {\mathbb R}^{n}$, see \cite[p. 174]{Fo} (this is an easy calculation).

This action induces the following action of $Sp(n,{\mathbb R})$ on $H_n$. Let us denote now by $(a,b,c)$ the element $((a+ib,a-ib),c)$ of $H_n$. Then the considered action is
\begin{equation*}g\cdot (a,b,c):=(g\cdot (a,b),c)\end{equation*} for each $g\in Sp(n,{\mathbb R})$.

Of course, the preceding action can be differentiated to an action of $Sp(n,{\mathbb R})$ on ${\mathfrak h}_n$. If we denote by $(u,v,c)$ the element $[(u+iv,u-iv),c]$ of ${\mathfrak h}_n$ where
$u,v \in {\mathbb R}^n$ and $c\in {\mathbb R}$, then we have
\begin{equation*}g\cdot (u,v,c):=(g\cdot (u,v),c), \quad g\in Sp(n,{\mathbb R}).\end{equation*} 
More precisely, let $g=\bigl(\begin{smallmatrix} A&B\\
C&D \end{smallmatrix}\bigr)\in Sp(n,{\mathbb R})$. Then we have
\begin{equation*}g\cdot (u,v,c):=(Au+Bv,Cu+Dv,c).\end{equation*} 

Now, we denote $\rho':=\rho'_{\lambda}$. For each $g\in Sp(n,{\mathbb R})$, let $\rho'_{g}$ be the representation of $H_n$ on $L^2({\mathbb R}^{n})$ defined by $\rho'_{g}(h)=\rho'(g\cdot h)$ for each $h\in H_n$.
As in Section \ref{sec:3}, for each $g\in Sp(n,{\mathbb R})$, there exists a unitary operator $A'_g$
on $L^2({\mathbb R}^{n})$ such that 
\begin{equation}\label{eq:interS} \rho'(g\cdot h)A'_g=A'_g\rho'(h)\end{equation}  for each $h\in H_n$. Then we also have
\begin{equation}\label{eq:interSd} d\rho'(g\cdot X)A'_g=A'_gd\rho'(X),\quad X\in {\mathfrak h}_n, \end{equation} 
on a suitable subspace of $L^2({\mathbb R}^{n})$.

Let $g\in Sp(n,{\mathbb R})$. We aim to find $A'_g$ as an integral operator, that is, we assume that there exists a function $a_g(x,y)$ such that we have
\begin{equation*}(A'_g\phi)(x)=\int_{{\mathbb R}^{n}}\, a_g(x,y)\phi(y)\,dy  \end{equation*}
for each $\phi$ in the Schwartz space ${\mathcal S}({\mathbb R}^{n})$.

\begin{proposition}\label{propkerS} Let $g\in Sp(n,{\mathbb R})$. Then $A'_g$ satisfies
(\ref{eq:interSd}) if its integral kernel $a_g(x,y)$ satisfies the relation 
\begin{equation}\label{eq:kerS} \begin{split}-i{\lambda}(Cu+Dv)&x\,a_g(x,y)-\sum_{j=1}^n(Au+Bv)_j
\frac{\partial}{\partial x_j}a_g(x,y)\\
&=-i{\lambda}vy\,a_g(x,y)+\sum_{j=1}^nu_j\frac{\partial}{\partial 
{y}_j}a_g(x,y).\end{split} \end{equation}
Here the derivatives of $a_g$ are interpreted when necessary as distributions.
\end{proposition}

\begin{proof} Let $g=\bigl(\begin{smallmatrix} A&B\\
C&D \end{smallmatrix}\bigr) \in Sp(n,{\mathbb R})$ and $X=(u,v,0)\in {\mathfrak h}_n$.
Then, on the one hand, we have
\begin{align*}
(d\rho'(g&\cdot X)A'_g\phi )(x)=-i{\lambda}(Cu+Dv)x(A'_g\phi)(x)-\sum_{j=1}^n(Au+Bv)_j
\frac{\partial}{\partial x_j}(A'_g\phi)\\
=&-i{\lambda}(Cu+Dv)x\,\int_{{\mathbb R}^{n}} a_g(x,y)\phi(y)\,dy-\sum_{j=1}^n(Au+Bv)_j
\int_{{\mathbb R}^{n}}\frac{\partial}{\partial x_j}a_g(x,y)\phi(y)dy
\end{align*} for each $\phi \in {\mathcal S}({\mathbb R}^{n})$.

On the other hand, we also have
\begin{align*}(A'_g&d\rho'(X)\phi)(x)=\int_{{\mathbb R}^{n}}a_g(x,y)(d\rho'(X)\phi)(y)dy
\\&=\int_{{\mathbb R}^{n}}a_g(x,y)\left(-i\lambda vy\phi(y)-\sum_{j=1}^n u_j\frac{\partial \phi}{\partial y_j}\right)\,dy
\end{align*} for each $\phi \in {\mathcal S}({\mathbb R}^{n})$ and, integrating by parts, we obtain
\begin{equation*}(A'_gd\rho'(X)\phi)(x)=\int_{{\mathbb R}^{n}}\left(-i\lambda vy\,a_g(x,y)+\sum_{j=1}^n u_j\frac{\partial}{\partial y_j}a_g(x,y)\right)\phi(y)dy,
\end{equation*} hence the result.
\end{proof}

\begin{proposition}\label{propformag}
Let $g=\bigl(\begin{smallmatrix} A&B\\
C&D \end{smallmatrix}\bigr) \in Sp(n,{\mathbb R})$. Then we have
\begin{equation*}a_g(x,y)=c'_g\, \exp \left( -i\lambda \left( \tfrac{1}{2}x(DB^{-1}x)-(B^{-1}x)y+
\tfrac{1}{2}y(BA^{-1}y)\right)\right)
\end{equation*} where $c'_g\in {\mathbb C}$.
\end{proposition}

\begin{proof} Let $g=\bigl(\begin{smallmatrix} A&B\\
C&D \end{smallmatrix}\bigr) \in Sp(n,{\mathbb R})$. Write
\begin{equation*}a_g(x,y) =\exp \alpha_g(x,y).\end{equation*}
Equation (\ref{eq:kerS}) then becomes
\begin{equation}\label{eq:kerSbis} \begin{split}-i{\lambda}(Cu+Dv)&x-\sum_{j=1}^n(Au+Bv)_j
\frac{\partial}{\partial x_j}\alpha_g(x,y)\\
&=-i{\lambda}vy+\sum_{j=1}^nu_j\frac{\partial}{\partial 
{y}_j}\alpha_g(x,y).\end{split} \end{equation}

For $\psi\in C^{\infty}({\mathbb R}^{2n})$, we use the notation
\begin{equation*}\nabla_x\psi=\sum_{j=1}^n\,\frac{\partial \psi}{\partial 
{x}_j}e_j;\quad \quad \nabla_y\psi=\sum_{j=1}^n\,\frac{\partial \psi}{\partial 
{y}_j}e_j.\end{equation*}

Now take $u=0$ in (\ref{eq:kerSbis}). This gives
\begin{equation*}i\lambda D^tx+B^t\nabla_x \alpha_g=i\lambda y.\end{equation*}
Then we have
\begin{equation}\label{521}\nabla_x \alpha_g=i\lambda (B^t)^{-1}(-D^tx+y).
\end{equation}
Thus, for each $j=1,2,\ldots,n$, we get
\begin{align*}\frac{\partial \alpha_g}{\partial {x}_j}&=-i\lambda e_j (B^t)^{-1}(D^tx-y)\\
&=-i\lambda (e_j(DB^{-1})^tx-e_j(B^t)^{-1}y)\\
&=-i\lambda \frac{\partial}{\partial 
{x}_j}\left( \tfrac{1}{2}x(DB^{-1}x)-x (B^t)^{-1}y\right)\\
\end{align*}
since $DB^{-1}$ is symmetric. Consequently, there exists a function $\gamma'_g(y)$ such that
\begin{equation}\label{522}\alpha_g(x,y)=-i\lambda \left( \tfrac{1}{2}x(DB^{-1}x)-x (B^t)^{-1}y\right)+\gamma'_g(y).
\end{equation}
On the other hand, we can take $v=0$ in (\ref{eq:kerSbis}). This gives
\begin{equation*}-i\lambda C^tx-A^t(\nabla_x\alpha_g)=\nabla_y \alpha_g.
\end{equation*} Then, in the left-side hand of the preceding equation, we can replace $\nabla_x\alpha_g$  by its expression taken from (\ref{521}) and, in the right-side hand, we can replace $\alpha_g$ by its expression given by  (\ref{522}). Thus we obtain
\begin{equation*}-i\lambda (C^t-A^t(B^t)^{-1}D^t+B^{-1})x-i\lambda(B^{-1}A)^t y=\nabla_y\gamma'_g.
\end{equation*}
But $B^{-1}A$ and $DB^{-1}$ are symmetric and
\begin{equation*}C^t-A^t(B^t)^{-1}D^t+B^{-1}=C^t-A^tDB^{-1}+B^{-1}=
(C^tB-A^tD+I_n)B^{-1}=0 \end{equation*} hence

\begin{equation*}\nabla_y \gamma'_g=-i\lambda B^{-1}Ay\end{equation*}
and
\begin{equation*}\frac{\partial \gamma'_g}{\partial {y}_j}=-i\lambda e_j (B^{-1}Ay)=-i\frac{\lambda}{2}\frac{\partial }{\partial {y}_j}(y(B^{-1}Ay))\end{equation*} for $j=1,2,\ldots,n$. This implies that there exists a scalar $\epsilon'_g$
such that 
\begin{equation*}\gamma'_g(y)=-i\tfrac{\lambda}{2}y(B^{-1}Ay)+\epsilon'_g.\end{equation*}
Replacing in (\ref{522}), we obtain the desired result. 
\end{proof}

\begin{proposition} \label{propcg} Let $g=\bigl(\begin{smallmatrix} A&B\\
C&D \end{smallmatrix}\bigr) \in Sp(n,{\mathbb R})$. Then the operator $A'_g$ of $L^2({\mathbb R}^n)$ with integral kernel
\begin{equation*}a_g(x,y)=c'_g\, \exp \left( -i\lambda \left( \tfrac{1}{2}x(DB^{-1}x)-(B^{-1}x)y+
\tfrac{1}{2}y(BA^{-1}y)\right)\right),
\end{equation*} where $c'_g\in {\mathbb C}$, is unitary if and only if we have
\begin{equation*}\vert c'_g\vert=\left(\frac{\lambda}{2\pi}\right)^{n/2}\vert \Det (B)\vert^{-1/2}.\end{equation*}
\end{proposition}

\begin{proof} Let $g=\bigl(\begin{smallmatrix} A&B\\
C&D \end{smallmatrix}\bigr) \in Sp(n,{\mathbb R})$. For each $\phi \in {\mathcal S}({\mathbb R}^n)$, we have
\begin{align*}(A'_g&(A'_g)^{\ast}\phi)(x)=\int_{{\mathbb R}^n}a_g(x,y)((A'_g)^{\ast}\phi)(y)dy\\=&\int_{{\mathbb R}^{2n}}a_g(x,y)\overline{a_g(u,y)}\phi(u)\,dydu\\
=&\vert c'_g\vert^2\, \exp(-i\tfrac{\lambda}{2}xDB^{-1}x)\\
\times & \int_{{\mathbb R}^{2n}}\exp\left(i\lambda B^{-1}(x-u)y\right)
\exp \left(i\tfrac{\lambda}{2}u(DB^{-1}u)\right)\,\phi(u)dy\,du\\
=&\vert c'_g\vert^2\,{\lambda}^{-n} \vert \Det B\vert\, \exp(-i\tfrac{\lambda}{2}xDB^{-1}x)\\ \times & \int_{{\mathbb R}^{2n}}e^{-ivy}\exp\left(i\tfrac{\lambda}{2}(x+{\lambda}^{-1}Bv)DB^{-1}(x+{\lambda}^{-1}Bv)\right)\,\phi(x+{\lambda}^{-1}Bv)\,dydv
\end{align*} after the change of variables defined by $v=-\lambda B^{-1}(x-u)$.

Recall that the Fourier transform $F$ defined on ${\mathcal S}({\mathbb R}^n)$ by
\begin{equation*}(F\phi)(x)=\int_{{\mathbb R}^{n}}e^{-ixy}\,\phi(y)\,dy\end{equation*}
can be extended to tempered distributions and that we have $F(1)=(2\pi)^n \delta$ where
$1$ is the constant function on ${\mathbb R}^n$ equal to $1$ and $\delta$ is the Dirac distribution. Then we have
\begin{align*}(A'_g&(A'_g)^{\ast}\phi)(x)=
\vert c'_g\vert^2\,{\lambda}^{-n} \vert \Det B\vert\, \exp(-i\tfrac{\lambda}{2}xDB^{-1}x)\\ \times & \int_{{\mathbb R}^{n}}F(1)(v)\exp\left(i\tfrac{\lambda}{2}(x+{\lambda}^{-1}Bv)DB^{-1}(x+{\lambda}^{-1}Bv\right)\,\phi(x+{\lambda}^{-1}Bv)\,dv\\
=&\vert c'_g\vert^2\,\left(\tfrac{2\pi}{\lambda}\right)^{n} \vert \Det B\vert\, \exp(-i\tfrac{\lambda}{2}xDB^{-1}x)\exp(i\tfrac{\lambda}{2}xDB^{-1}x)\phi(x).
\end{align*}
But $(A'_g(A'_g)^{\ast}\phi)(x)=\phi(x)$ hence $\vert c'_g\vert^2\,\left(\tfrac{2\pi}{\lambda}\right)^{n} \vert \Det B\vert=1$.
\end{proof}

\begin{proposition}\label{propcgsuite}  Let $g=\bigl(\begin{smallmatrix} A&B\\
 C&D \end{smallmatrix}\bigr)\in Sp(n,{\mathbb R})$,  $g'=\bigl(\begin{smallmatrix} A'&
B'\\C'&D' \end{smallmatrix}\bigr)\in Sp(n,{\mathbb R})$and let $g''=gg'=\bigl(\begin{smallmatrix} A''&B''\\
 C''& D'' \end{smallmatrix}\bigr)$. We denote by $A'_g$
the unitary operator on $L^2({\mathbb R}^n)$ with integral kernel
\begin{equation*}a_g(x,y)=d_g\,\left(\tfrac{\lambda}{2\pi}\right)^{n/2} \vert \Det B\vert^{-1/2}\,
\exp \left( -i\lambda \left( \tfrac{1}{2}x(DB^{-1}x)-(B^{-1}x)y+\tfrac{1}{2}y(B^{-1}Ay)\right)\right),
\end{equation*} where $d_g$ is a unit complex number.
\begin{enumerate}
\item There exists
a unit complex number $\alpha(g,g')$ such that \begin{equation}\label{eq:alpha} A'_{g''}=\alpha(g,g')A'_gA'_{g'}; \end{equation}
\item We have 
\begin{equation*}d_{g''}d_{g}^{-1}d_{g'}^{-1}=\alpha(g,g')\vert \Det(B^{-1}B''B'^{-1})\vert^{1/2}(\Det(iB^{-1}B''B'^{-1}))^{-1/2};
\end{equation*}
\item If for each $g\in Sp(n,{\mathbb R})$, we take $d_g$ such that
\begin{equation*}d_g^2=i^n \vert \Det B\vert (\Det B)^{-1}\end{equation*}
then we have $\alpha(g,g')=\pm 1$. In other words, if we write $\sigma'(g)$ instead of  $A'_g$ in this case, then we have $\sigma'(gg')=\pm \sigma'(g)\sigma'(g')$.\end{enumerate}
\end{proposition}

\begin{proof} (1) This is an immediate consequence of Schur's lemma.

(2) The integral kernel of  $A'_gA'_{g'}$ is given by
\begin{align*}k(&x,y):=\int_{{\mathbb R}^{n}}a_g(x,u)\,a_{g'}(u,y)\,du\\
=&d_gd_{g'}\left(\tfrac{\lambda}{2\pi}\right)^{n}\vert \Det B\vert^{-1/2}\vert \Det B'\vert^{-1/2}\exp(-i\tfrac{\lambda}{2}xDB^{-1}x)\exp(-i\tfrac{\lambda}{2}yB'^{-1}A'y)
\\
\times &
\int_{{\mathbb R}^{n}}\exp(-i\tfrac{\lambda}{2}u(B^{-1}A+D'B'^{-1})u)
\exp(i\lambda (B^{-1}x+(B'^{-1})^ty)u)\,du.
\end{align*}
The preceding integral can be computed by using (\ref{eq:Gauss}) via Gaussian regularization. Note that
\begin{equation*}B^{-1}A+D'B'^{-1}=B^{-1}(AB'+BD')B'^{-1}=B^{-1}B''B'^{-1}.
\end{equation*} By equating the coefficients of the exponential 
on both sides of the equality
\begin{equation*}a_{g''}(x,y)=\alpha(g,g')k(x,y)\end{equation*}
which is a reformulation in terms of integral kernels of (\ref{eq:alpha}), we get
\begin{equation*}d_{g''}\vert \Det B''\vert^{-1/2}=\alpha(g,g')\vert \Det B\vert^{-1/2}\vert \Det B'
\vert^{-1/2}d_gd_{g'}(i^n\Det(B^{-1}B''B'^{-1}))^{-1/2},
\end{equation*} hence the result.

(3) We have
\begin{equation*}(d_{g''}d_g^{-1}d_{g'}^{-1})^2=\alpha(g,g')^2\vert \Det(B^{-1}B''B'^{-1}) \vert (\Det(iB^{-1}B''B'^{-1}))^{-1}.
\end{equation*} By taking into account the values chosen for $d_g, d_{g'}$ and $d_{g''}$, the result follows.
\end{proof}

The scalar $d_g$ in Proposition \ref{propcgsuite} is connected to the so-called \textit{Maslov index}. See, for instance, \cite{Gos}. We will not discuss this point in detail here in order to keep this note at a quite elementary level.

\section{ The Bargmann transform} \label{sec:6}
In this section, we first introduce the Bargmann transform which is a unitary intertwining operator between the Fock realization and the Schr\"odinger realization of a generic representation of $H_n$. Then, following a suggestion of \cite[p. 184]{Fo}, we use the
Bargmann transform in order to relate the metaplectic kernel for the Schr\"odinger to that
for the Fock model. This gives an alternate method for constructing the metaplectic kernel in the Schr\"odinger model.

The Bargmann transform $\mathcal B$ is the operator from $L^2({\mathbb R}^n)$ onto $\mathcal F_{\lambda}$
defined by
\begin{equation*}({\mathcal B}\phi)(z):=\int_{{\mathbb R}^n}\,b(z,x)\phi(x)\, dx\end{equation*}
where
\begin{equation*}b(z,x)=\left(\tfrac{\lambda}{\pi}\right)^{n/4}\,\exp\left(
-\tfrac{\lambda}{4}z^2+\lambda zx-\tfrac{\lambda}{2}x^2 \right).\end{equation*}
It is well known that $\mathcal B$ is unitary and that, for each $h\in H_n$, we have
\begin{equation*}{\mathcal B}\rho'(h)=\rho(h){\mathcal B},
\end{equation*} see, for instance, \cite{Barg, Fo, Tay}. 

Remark that we can recover the integral kernel $b(z,x)$ of $\mathcal B$ by differentiating the preceding relation with respect to $h$ and then applying the method used in the present note, see \cite[p. 60]{Tay}.

Since $\mathcal B$ is unitary, we can write, for each $\phi\in L^2({\mathbb R}^n)$ and each $f\in {\mathcal F}_{\lambda}$,
\begin{equation*}\langle {\mathcal B}\phi, f \rangle_{\mathcal F_{\lambda}} =
\langle \phi, {\mathcal B}^{-1}f \rangle_{L^2({\mathbb R}^n)}.\end{equation*}
This gives the following expression for ${\mathcal B}^{-1}$:
\begin{equation*}({\mathcal B}^{-1}f)(x):=\int_{{\mathbb C}^n}\,\overline{b(z,x)}f(z)\, e^{-\lambda \vert
z\vert^2/2}\,d\mu_{\lambda} (z).\end{equation*}

We need later the following formulas for the inverse matrix of a block matrix.

\begin{lemma}\label{leminvmat} Let $A,B,C$ and $D$ be $n\times n$ complex matrices
and let $M=\left(\begin{matrix} A&B\\
C&D \end{matrix}\right)$.
\begin{enumerate}
\item Assume that $D$ is invertible and let $A_1=A-BD^{-1}C$. If $A_1$ is also invertible then $M$ is invertible and we have
\begin{equation*}M^{-1}=\begin{pmatrix} A_1^{-1}&-A_1^{-1}BD^{-1}\\
-D^{-1}CA_1^{-1}&D^{-1}CA_1^{-1}BD^{-1}+D^{-1}\end{pmatrix} \end{equation*}
(First inversion formula).

\item Assume that $A$ is invertible and let $A_2=D-CA^{-1}B$. If $A_2$ is also invertible then $M$ is invertible and we have
\begin{equation*}M^{-1}=\begin{pmatrix} A^{-1}+A^{-1}BA_2^{-1}CA^{-1}&-A^{-1}BA_2^{-1}\\-A_2^{-1}CA^{-1}&A_2^{-1}\end{pmatrix}\end{equation*}
(Second inversion formula).
\end{enumerate} \end{lemma}

\begin{proof} The first inversion formula is a consequence of the triangular decomposition
\begin{equation*}M=\begin{pmatrix} A&B\\
C&D \end{pmatrix}=\begin{pmatrix} I_n&BD^{-1}\\
0&I_n \end{pmatrix}=\begin{pmatrix} A_1&0\\
0&D \end{pmatrix}\begin{pmatrix} I_n&0\\
D^{-1}C&D \end{pmatrix}\end{equation*}
whereas the second inversion formula derives from the triangular decomposition
\begin{equation*}M=\begin{pmatrix} A&B\\
C&D \end{pmatrix}=\begin{pmatrix} I_n&0\\
CA^{-1}&I_n \end{pmatrix} \begin{pmatrix} A&0\\
0&A_2 \end{pmatrix}\begin{pmatrix} I_n&A^{-1}B\\
0&I_n \end{pmatrix}.\end{equation*}
\end{proof}

In Section \ref{sec:4}, for each $k=\left(\begin{smallmatrix} P&Q\\
{\bar Q}&{\bar P} \end{smallmatrix}\right)\in S$, we considered the operator $\sigma(k)$
on ${\mathcal F}$ with integral kernel
\begin{equation*} b_k(z,w)=(\Det P)^{-1/2}\,\exp\left( \frac{\lambda}{4}z{\bar Q}P^{-1}z+\frac{\lambda}{2}(P^{-1}z)\bar{w}-\frac{\lambda}{4}\bar{w}(P^{-1}Q\bar{w})\right).\end{equation*}

Now, for each $k\in S$, we define
\begin{equation}\label{eq:sigma1}\sigma_1(k)={\mathcal B}^{-1}\sigma(k){\mathcal B}
\end{equation} which is an operator on $L^2({\mathbb R}^n)$.

\begin{proposition}\label{propBk} Let $k=\left(\begin{smallmatrix} P&Q\\
{\bar Q}&{\bar P} \end{smallmatrix}\right)\in S$. Let $g=U^{-1}kU=\left(\begin{smallmatrix} A&B\\C&D\end{smallmatrix}\right)\in Sp(n,{\mathbb R})$
where $U=\left(\begin{smallmatrix} I_n&iI_n\\I_n&-iI_n\end{smallmatrix}\right)$, see Section \ref{sec:3}. Then the integral kernel $B_k(x,y)$ of $\sigma_1(k)$ is given by
\begin{align*}B_k&(x,y)=\left( \tfrac{\lambda}{2\pi}\right)^{n/2}(\Det P)^{-1/2}(i^n(\Det P)(\Det B)^{-1})^{-1/2}\\
\times&\exp\left(-\tfrac{\lambda}{2}i(xDB^{-1}x-2yB^{-1}x+yB^{-1}Ay)\right).
\end{align*}
\end{proposition}

\begin{proof} In order to simplify the notation, we write $k=\left(\begin{smallmatrix} P&Q\\R&S \end{smallmatrix}\right)$ instead of $\left(\begin{smallmatrix} P&Q\\
{\bar Q}&{\bar P} \end{smallmatrix}\right)$. Translating (\ref{eq:sigma1}) in terms of kernels, we get
\begin{equation*}B_k(x,y)=\int_{{\mathbb C}^{2n}}\overline{b(z,x)}b_k(z,w)b(w,y)e^{-\tfrac{\lambda}{2}(\vert z\vert^2+\vert w\vert^2)}\,d\mu_{\lambda}(z)d\mu_{\lambda}(w).
\end{equation*}
This gives
\begin{equation*}B_k(x,y)=\left( \tfrac{\lambda}{\pi}\right)^{n/2}
\left( \tfrac{\lambda}{2\pi}\right)^{2n}(\Det P)^{-1/2}\exp\left(-\tfrac{\lambda}{2}(x^2+y^2)\right)\, I(x,y)
\end{equation*}
where
\begin{align*}&I(x,y):=\int_{{\mathbb C}^{2n}}\,\exp\left(-\tfrac{\lambda}{4}({\bar z}^2+w^2)+\lambda({\bar z}x+wy)\right)\\
&\times 
\exp\left(\tfrac{\lambda}{4}(zRP^{-1}z+2(P^{-1}z){\bar w}-{\bar w}(P^{-1}Q{\bar w})\right)
\,e^{-\tfrac{\lambda}{2}(\vert z\vert^2+\vert w\vert^2)}\,dm(z,w),\end{align*}
$dm(z,w)$ denoting the Lebesgue measure on ${\mathbb C}^{2n}$.

Writing $X=(x,y)\in {\mathbb R}^{2n}$ and $Z=(z,{\bar w})\in {\mathbb C}^{2n}$,
we can express $I(x,y)$ as follows. Let
\begin{equation*} E:=-\begin{pmatrix} RP^{-1}&(P^t)^{-1}\\P^{-1}&-P^{-1}Q\end{pmatrix}. \end{equation*}
Then we have
\begin{equation*}I(x,y)=\int_{{\mathbb C}^{2n}}\,\exp\left(-\tfrac{\lambda}{4}{\bar Z}^2+\lambda {\bar Z}X-\tfrac{\lambda}{4}ZEZ\right) \,e^{-\tfrac{\lambda}{2}\vert Z\vert^2}dm(Z).\end{equation*}
We are now in position to apply Lemma \ref{lemgauss}. We consider the matrices
\begin{equation*}M:=\tfrac{\lambda}{4}\begin{pmatrix} E&I_{2n}\\I_{2n}&I_{2n}\end{pmatrix},\quad U':=\begin{pmatrix} I_{2n}&iI_{2n}\\I_{2n}&-iI_{2n}\end{pmatrix},\quad N=U'^tMU'. \end{equation*}
The inverse matrix of $M$ is 
\begin{equation*}M^{-1}=\tfrac{4}{\lambda}\begin{pmatrix} (E-I_{2n})^{-1}&-(E-I_{2n})^{-1}\\-(E-I_{2n})^{-1}&E(E-I_{2n})^{-1}\end{pmatrix}.
\end{equation*}
Thus we have
\begin{equation*}\tfrac{1}{4}\left(\begin{matrix}0&\lambda X\end{matrix}\right)M^{-1}
\left(\begin{matrix}0 \\ \lambda X\end{matrix}\right)=\lambda XE(E-I_{2n})^{-1}X=\lambda (X^2+X(E-I_{2n})^{-1}X)\end{equation*}
hence
\begin{equation*}I(x,y)=\pi^{2n}(\Det N)^{-1/2}\exp\left(\lambda (X^2+X(E-I_{2n})^{-1}X)\right)
\end{equation*}
and
\begin{align*}B_k&(x,y)=\pi^{2n}\left( \tfrac{\lambda}{\pi}\right)^{n/2}
\left( \tfrac{\lambda}{2\pi}\right)^{2n}(\Det P)^{-1/2}(\Det N)^{-1/2}\\
&\times 
\exp\left(\tfrac{\lambda}{2}(x^2+y^2)-\lambda \left(\begin{matrix}x&y\end{matrix}\right)(I_{2n}-E)^{-1}\left(\begin{matrix}x\\y\end{matrix}\right)\right).
\end{align*}
Now we compute
\begin{equation*}F:=-I_{2n}+2(I_{2n}-E)^{-1}=-I_{2n}+2\begin{pmatrix} RP^{-1}+I_n&(P^t)^{-1}\\P^{-1}&I_n-P^{-1}Q\end{pmatrix}^{-1}.
\end{equation*}
We write
\begin{equation*}F=\begin{pmatrix} F_{11}&F_{12}\\F_{12}&F_{22}\end{pmatrix}.
\end{equation*}
First, we compute $F_{12}$. By the first inversion formula (see Lemma \ref{leminvmat}),
we have
\begin{equation*}-\tfrac{1}{2}F_{12}=\left(RP^{-1}+I_n-(P^t)^{-1}(I_n-P^{-1}Q)^{-1}P^{-1}\right)^{-1}(P^t)^{-1}(I_n-P^{-1}Q)^{-1}.
\end{equation*} It would be preferable here to use the inverse matrix. By applying successively the relations $QP^t=PQ^t$, $P^tR=R^tP$ and $I_n+Q^tR=S^tP$, we have
\begin{align*}
-\left(\tfrac{1}{2}F_{12}\right)^{-1}&
=(I_n-P^{-1}Q)P^t\left(RP^{-1}+I_n-(P(I_n-P^{-1}Q)P^t)^{-1}\right)\\
&=(P^t-Q^t)(RP^{-1}+I_n)-P^{-1}\\
&=P^tRP^{-1}+P^t-Q^tRP^{-1}-Q^t-P^{-1}\\
&=R^t+P^t-Q^t-(I_n+Q^tR)P^{-1}\\
&=(P-Q+R-S)^t.
\end{align*}
Recall that the relation $g=U^{-1}kU$ can be written as
\begin{equation*}\begin{pmatrix} A&B\\C&D\end{pmatrix} =\tfrac{1}{2}
\begin{pmatrix} P+Q+R+S&i(P-Q+R-S)\\-i(P+Q-R-S)&P-Q-R+S\end{pmatrix}.
\end{equation*}
Then we get \begin{equation*}-\left(\tfrac{1}{2}F_{12}\right)^{-1}=\left(\tfrac{2}{i}B\right)^t \end{equation*} hence $F_{12}=-i(B^t)^{-1}$.

By using the first inversion formula again, we also have
\begin{align*}
F_{11}&=-I_n+2\left(-\tfrac{1}{2}F_{12}\right)(I_n-P^{-1}Q)P^t\\
&=-I_n+i(B^t)^{-1}(P^t-Q^t)\\
&=(B^t)^{-1}(-B^t+iP^t-iQ^t)\\
&=(B^t)^{-1}(-\tfrac{i}{2}(P^t-Q^t+R^t-S^t)+i(P^t-Q^t))\\
&=\tfrac{i}{2}(B^t)^{-1}(P^t-Q^t-R^t+S^t)\\
&=i(B^t)^{-1}D^t=i(DB^{-1})^t=iDB^{-1}.
\end{align*}

Now, by the second inversion formula, we get
\begin{align*}
F_{22}&=-I_n+2(I_n-P^{-1}Q-P^{-1}(RP^{-1}+I_n)^{-1}(P^t)^{-1})^{-1}\\
&=-I_n+2(I_n-P^{-1}Q-(P^t(RP^{-1}+I_n)P)^{-1})^{-1}\\
&=-I_n+2(I_n-P^{-1}Q-(P^tR+P^tP)^{-1})^{-1}\\
&=-I_n+2((P^tR+P^tP)(I_n-P^{-1}Q)-I_n)^{-1}(P^tP+P^tR)\\
&=-I_n+2(P^tR+P^tP-P^tRP^{-1}Q-P^tQ-I_n)^{-1}(P^tP+P^tR).
\end{align*}
Since $P^tR=R^tP$, we find
\begin{equation*}F_{22}=-I_n+2(P^tR+P^tP-P^tQ-R^tQ-I_n)^{-1}(P^tP+P^tR).
\end{equation*}
By using $R^tQ+I_n=P^tS$, we get
\begin{align*}
F_{22}&=-I_n+2(P^t(R+P-Q-S))^{-1}(P^tP+P^tR)\\
&=-I_n+2(R+P-Q-S)^{-1}(P+R)\\
&=(R+P-Q-S)^{-1}(-(R+P-Q-S)+2P+2R)\\
&=iB^{-1}A.
\end{align*}
To summarize, we have that
\begin{equation*}F=\begin{pmatrix} iDB^{-1}&-i(B^t)^{-1}\\-iB^{-1}&iB^{-1}A\end{pmatrix} 
\end{equation*}
and the exponential in $B_k(x,y)$ is 
\begin{equation*} \exp\left(-\tfrac{\lambda}{2}i(xDB^{-1}x-2yB^{-1}x+yB^{-1}Ay)\right).
\end{equation*}
It remains to compute the coefficient in front of the exponential. We remark that
\begin{equation*} \Det (N)=(\Det U')^2\Det(M)=2^{4n}\left(\tfrac{\lambda}{4}\right)^{4n}
\Det(E-I_{2n}).
\end{equation*}
But we have
\begin{align*}
\begin{pmatrix} RP^{-1}+I_n&(P^t)^{-1}\\P^{-1}&I_n-P^{-1}Q\end{pmatrix}
&\begin{pmatrix} P&Q\\0&I_n\end{pmatrix}=\begin{pmatrix} P+R&RP^{-1}Q+Q+(P^t)^{-1}\\I_n&I_n\end{pmatrix}\\&=\begin{pmatrix}P+R&Q+S\\I_n&I_n\end{pmatrix}
\end{align*} since
\begin{equation*}
RP^{-1}Q+(P^t)^{-1}=(RP^{-1}QP^t+I_n)(P^t)^{-1}=(RQ^t+I_n)(P^t)^{-1}=SP^t(P^t)^{-1}=S. \end{equation*}
This implies that
\begin{align*}
\Det(E-I_{2n})=&\Det\begin{pmatrix} RP^{-1}+I_n&(P^t)^{-1}\\P^{-1}&I_n-PQ^{-1}\end{pmatrix}\\
=&(\Det P)^{-1}\Det(P-Q+R-S)=\left(\tfrac{2}{i}\right)^{n}(\Det P)^{-1}\Det B.
\end{align*}
Finally, the coefficient in question is
\begin{equation*}\left(\tfrac{\lambda}{2\pi}\right)^{n/2}(\Det P)^{-1/2}
(i^{-n}(\Det P)^{-1}\Det B)^{-1/2}\end{equation*} as announced. This ends the proof.
\end{proof}
Now, we aim to compare $\sigma'$ (see Section \ref{sec:5}) to $\sigma_1$.

\begin{corollary} \label{corcompa} For each $k\in S$, we have $\sigma'(U^{-1}kU)=\pm \sigma_1(k)$. \end{corollary}

\begin{proof} Taking Proposition \ref{propcgsuite} and Proposition \ref{propBk}
into account, we have just to compare the coefficients
\begin{equation*}c_1:=\left(\tfrac{\lambda}{2\pi}\right)^{n/2}(i^n(\Det B)^{-1})^{1/2}
\end{equation*}
and
\begin{equation*}c_2:=\left(\tfrac{\lambda}{2\pi}\right)^{n/2}(\Det P)^{-1/2}
(i^{n}\Det P (\Det B)^{-1})^{1/2}.
\end{equation*}
Let $b:=\Det B$, $p=\Det P$, $u_1:=(b^{-1}i^n)^{1/2}$, $u_2:=p^{-1/2}(i^npb^{-1})^{1/2}$ and $a:=i^nb^{-1}$. Then we have $u_2=p^{-1/2}(pa)^{1/2}$. Consider the polar decompositions $p=\vert p\vert e^{i\theta}$ and $a=\vert a\vert e^{i\varphi}$ where
$\theta, \varphi \in )-\pi,\pi($. Then we have $u_1=a^{1/2}=\vert a\vert^{1/2} e^{i\varphi/2}$ and
\begin{equation*}u_2=\vert p\vert^{-1/2}e^{-i\theta/2}\vert pa\vert^{1/2}
e^{i\Arg(pa)/2}.\end{equation*}
Here $\Arg$ denotes the principal determination of the argument.
Thus we have
\begin{equation*}\frac{u_2}{u_1}=\exp\left( \tfrac{i}{2}(\Arg(pa)-\theta-\varphi)\right).
\end{equation*}
Since $\Arg(pa)=\theta+\varphi+2\pi m$ with $m\in {\mathbb Z}$, we get $u_2=\pm u_1$, hence the result.
\end{proof}

\section{Final remarks} \label{sec:7}
The formulas for the metaplectic kernel we have recovered in this note provide the starting point for numerous developments. Let us mention a few of them below.

{\bf 7.1} (Complex extension of the metaplectic representation) Brunet-Kramer
developed the complex extension in the Fock model \cite{BK}. They studied operators on the Fock space whose integral kernels are given by the formula of Proposition \ref{propck}  where  $\bigl(\begin{smallmatrix} P&Q\\
\bar Q&\bar P \end{smallmatrix}\bigr)\in S$ is replaced by  $\bigl(\begin{smallmatrix} P&Q\\Q'&P' \end{smallmatrix}\bigr)\in Sp(n,{\mathbb C})$. So they identified a semigroup of $Sp(n,{\mathbb C})$ on which the metaplectic representation extends analytically. Moreover, Howe approached the same phenomenom in the
Schr\"odinger model, constructing the oscillator semigroup concretely as Gaussian integral operators \cite{Ho}. Hilgert did the synthesis proving that the Brunet-Kramer Fock space semigroup and the Howe's $L^2$-semigroup are isomorphic \cite{Hil}. This last result is a generalization of Proposition \ref{propBk}.

{\bf 7.2} (Integral formulas for the degenerate case)
The formulas for the metaplectic kernel given in this note are available under some assumptions on the symplectic matrix. For instance, in the formula for $a_g(x,y)$ given in Proposition \ref{propcgsuite}, $g=\bigl(\begin{smallmatrix} A&B\\
 C&D \end{smallmatrix}\bigr)$ must be such that $\Det B\not= 0$. Morsche and Oonincx
introduced a method that avoids such restrictions \cite{MO}. Instead of considering
an integral over ${\mathbb R}^n$, they identified appropriate linear subspaces of 
${\mathbb R}^n$ determined by $B$ and formulated the metaplectic operator as an oscillatory integral over those subspaces. Thus no assumption on $B$ is required. Similar integral formulas were also obtained for generalized metaplectic operators, see \cite{CNR}. Nevertheless, precisely since the domain of the integrals depends on the symplectic matrix, the composition of two operators defined by such integral representations is less simple than in the non degenerate case.

{\bf 7.3} (Metaplectic representation) Apart from \cite{Fo}, a good study of the metaplectic representation can be found in \cite{Gos}. Another valuable references are \cite{Der1, Ner}. For a more representation-theoretic treatment, see \cite{KV, LV}. The metaplectic representation can be also studied from the perspective of geometric quantization \cite{Wood}.

{\bf 7.4} (Weyl correspondence) The Weyl correspondence $\mathcal W$ which maps a function $u$ on 
 ${\mathbb R}^{2n}$ to an operator ${\mathcal W}(u)$ on $L^2({\mathbb R}^n)$ (then we say that $u$ is the Weyl symbol of ${\mathcal W}(u)$)  is covariant with respect
to $\sigma'$, that is, we have
\begin{equation*}\mathcal W(L_gu)=\sigma'(g)\mathcal W(u)\sigma'(g)^{-1}\end{equation*} with the notation $L_g u(x,y)=u(g^{-1}\cdot (x,y))$,
see \cite{CR, Fo, Gos}.

On the other hand, we know explicit formulas for the Weyl symbol of $\sigma'(g)$ for $g\in Sp(n,{\mathbb R})$ and also for the Weyl symbol of $d\sigma'(X)$ for $X$ in the
Lie algebra $sp(n,{\mathbb R})$ of $Sp(n,{\mathbb R})$ \cite{CR}. In particular, ${\mathcal W}(d\sigma'(X))$ is a quadratic form for  $X\in sp(n,{\mathbb R})$.

These considerations have some interesting consequences, namely
\begin{enumerate}
\item We can obtain formulas for the Weyl symbol of an operator on $L^2({\mathbb R}^n)$ having a quadratic form as Weyl symbol \cite{Ho2};
\item Let $(x,y)\in {\mathbb R}^{2n}$. Denote by $\psi (x,y)$ the map defined by
\begin{equation*}\psi (x,y)(X)=-i{\mathcal W}(d\sigma'(X))(x,y)\end{equation*}
 which is an element of the dual $sp(n,{\mathbb R})^{\ast}$ of 
$sp(n,{\mathbb R})$. Denote by $\mathcal O$ the minimal (non trivial) coadjoint orbit
of $Sp(n,{\mathbb R})$. Then the image of $\psi:(x,y)\rightarrow \psi(x,y)$ is ${\mathcal O}\cup (0)$ and the restriction of $\psi$ to ${\mathbb R}^{2n}\setminus \{(0,0\}$ is a $2$-covering of $\mathcal O$ \cite{CaComp,CaMR}. This gives a connection between $\sigma'$ and $\mathcal O$
in the spirit of the Kirillov-Kostant method of orbits \cite{Kir};
\item The so-called Moyal star product is defined on $C^{\infty}$-functions on ${\mathbb R}^{2n}$
by
\begin{equation*} u\ast v={\mathcal W}^{-1}({\mathcal W}(u){\mathcal W}(v))\end{equation*} or, equivalently, by the expansion
\begin{equation*}u \ast v:=\sum_{l\geq 0}\frac{(-i/2)^l}{l!}P^l(u,v)\end {equation*}
where $P^l$ is the $l$-th tensor power of the Poisson bracket.

An important problem in deformation quantization theory is the computation of the star exponential of a function $u$
\begin{equation*}\exp_{\ast}(u)=\sum_{l\geq 0}\frac{1}{l!}u^{\ast,l}.\end{equation*}
Assuming that $u$ is a quadratic form, we can find $X\in sp(n,{\mathbb R})$ such that
${\mathcal W}(-iu)=d\sigma'(X)$ and then the relations
\begin{equation*}\exp_{\ast}(-iu)=\exp_{\ast}({\mathcal W}^{-1}(d\sigma'(X)))={\mathcal W}^{-1}(\exp(d\sigma'(X)))
={\mathcal W}^{-1}(\sigma'(\exp( X ))\end{equation*}
leads to an explicit formula for $\exp_{\ast}(-iu)$, see \cite{CaComp}.
\end{enumerate}

{\bf 7.5} (Complex Weyl correspondence) In the series of papers \cite{CaComp, CaExt, CaMR}, we obtain results similar to those of {\bf 7.4} in the Fock model by using a complex version of the Weyl calculus.


\begin{thebibliography}{2}

\bibitem{Barg} Bargmann V., Group representations on Hilbert spaces of analytic functions, Analytic methods in mathematical physics (Sympos., Indiana Univ., Bloomington, Ind., 1968), pp. 27--63. Gordon and Breach, New York, 1970.

\bibitem{Bern} Berndt R. and Schmidt R., Elements of the
representation theory of the Jacobi group, Progress in Mathematics
163, Birkh\"auser Verlag, Basel, 1998.

\bibitem{BK} Brunet M. and Kramer P., Complex extension of the representation of the symplectic group associated with the canonical commutation relations, Rep. Math. Phys. 17 (1980), 205--215.

\bibitem{CaPad} Cahen B., Berezin transform and Stratonovich-Weyl correspondence for the multi-dimensional Jacobi group,
Rend. Semin. Mat. Univ. Padova 136 (2016), 69--93.

\bibitem{CaComp} Cahen B.,  Complex Weyl symbols of metaplectic operators: an elementary approach, Rend. Istit. Mat. Univ. Trieste 55 (2023), Paper No. 5, 27 pp.

\bibitem{CaExt} Cahen B., Complex Weyl symbols of the extended metaplectic representation operators, Oper. Matrices 18, 2 (2024), 457--477.

\bibitem{CaMR} Cahen B., Complex Weyl correspondence and metaplectic representation of the Jacobi group, Math. Rep. (Bucur.) 28(78), no. 1--2,  (2026), 15--37.

\bibitem{CNR} Cordero E.,  Nicola F. and Rodino L.,  Integral representations for the class of generalized metaplectic operators, J. Fourier Anal. Appl. 21 (2015), 694--714.

\bibitem{CR} Combescure  M. and Robert D., Coherent states and applications in mathematical physics, Theoretical and Mathematical Physics, Springer, Dordrecht, 2012.

\bibitem{Der1} Derezi\'nski J. and Karczmarczyk M.,
Quantization of Gaussians, Kurasov, Pavel et al. (eds), Analysis as a tool in mathematical physics. Birkhäuser. Oper. Theory: Adv. Appl. 276 (2020), 277--304.

\bibitem{Gos} de Gosson M. A., Symplectic methods in harmonic analysis and in mathematical physics, Pseudo-Differential Operators. Theory and Applications, 7. Birkh\"auser-Springer Basel AG, Basel, 2011.

\bibitem{Fo} Folland B.,  Harmonic Analysis in Phase Space,
Princeton Univ. Press, 1989.

\bibitem{Hil} Hilgert J., A note on Howe's oscillator semigroup, Ann. Inst. Fourier (Grenoble) 39 (1989), 663--688.

\bibitem{Ho} Howe R., The oscillator semigroup, The mathematical heritage of Hermann Weyl (Durham, NC, 1987), 61--132, Proc. Sympos. Pure Math., 48, Amer. Math. Soc., Providence, RI, 1988.

\bibitem{Ho1} H\"ormander L., The analysis of linear partial
differential operators, Vol. 3, Section 18.5, Springer-Verlag,
Berlin, Heidelberg, New-York, 1985.

\bibitem{Ho2} H\"ormander L., Symplectic classification of quadratic forms, and general Mehler formulas, Math. Z. 219 (1995), 413--449.

\bibitem{KV} Kashiwara M. and Vergne M., On the Segal-Shale-Weil Representations and Harmonic Polynomials, Inventiones Math. 44 (1978), 1--47.

\bibitem{Itz} Itzykson C., Remarks on boson commutation rules, Comm. Math. Phys. 4 (1967), 92--122.

\bibitem{LV} Lion G. and Vergne M., The Weil representation, Maslov index and theta series, Progress in Mathematics, 6. Birkh\"auser, Boston, MA, 1980.

\bibitem{Kir} Kirillov A. A., Lectures on the Orbit Method, Graduate
Studies in Mathematics Vol. 64, American Mathematical Society,
Providence, Rhode Island, 2004.

\bibitem{MO} ter Morsche H. and Oonincx  P. J.,  On the integral representations for metaplectic operators,  J. Fourier Anal. Appl. 8 (2002), 245--257.

\bibitem{Ne} Neeb K-H., Holomorphy and Convexity in Lie Theory, de
Gruyter Expositions in Mathematics, Vol. 28, Walter de Gruyter,
Berlin, New-York 2000.

\bibitem{Ner} Neretin Y. A, Lectures on Gaussian integral operators and classical groups, EMS Series of Lectures in Mathematics. European Mathematical Society (EMS), Z\"urich, 2011. 

\bibitem{Tay} Taylor M. E., Noncommutative Harmonic Analysis, Mathematical
Surveys and Monographs 22, American Mathematical Society,
Providence, Rhode Island 1986.

\bibitem{Wood}  Woodhouse N. M. J, Geometric quantization, Second edition, Oxford Mathematical Monographs. Oxford Science Publications. The Clarendon Press, Oxford University Press, New York, 1992. 

\end{thebibliography}
\end{document}